\documentclass[12pt,fleqn]{article}

\usepackage{amsmath}
\usepackage{amsthm}
\usepackage{amssymb}
\usepackage{amsfonts}
\usepackage{graphicx}
\usepackage{hyperref}
\usepackage{color}

\newtheorem{thm}{Theorem}

\newtheorem{prop}{Proposition}

\theoremstyle{remark}
\newtheorem{rmk}{Remark}

\theoremstyle{definition}

 \DeclareMathOperator\re{{Re}}
 \numberwithin{equation}{section}

\newcommand{\D}{\displaystyle}

\numberwithin{equation}{section}

\newcounter{comment}
\def\a{\alpha}
\def\b{\beta}
\def\k{\kappa_{0}}
\def\K{\kappa_{\infty}}

\begin{document}

\title{Existence and Uniqueness of Tronqu\'{e}e Solutions of the Third and Fourth Painlev\'{e} Equations}
\date{\today}
\author{Y. Lin$^\ast$, D. Dai$^\dag$ and P. Tibboel$^\ddag$}
\maketitle

\begin{abstract}
  It is well-known that the first and second Painlev\'e equations admit solutions characterised by divergent asymptotic expansions near infinity in specified sectors of the complex plane.  Such solutions are pole-free in these sectors and called tronqu\'ee solutions by Boutroux. In this paper, we show that similar solutions exist for the third and fourth Painlev\'e equations as well.
\end{abstract}

%%%%%%%%%%%%%%%%%%%%%%%%%%%%%%%%%%%%%%%%%%%%%%%%%%%%%%%%%%%%%%%%%%

\vspace{5mm}



%%%%%%%%%%%%%%%%%%%%%%%%%%%%%%%%%%%%%%%%%%%%%%%%%%%%%%%%%%%%%%%%%%

\vspace{5mm}

\hrule width 65mm

\vspace{2mm}

\begin{description}

\item \hspace*{3.8mm}$\ast$ Liu Bie Ju Centre for Mathematical Sciences, City University of
Hong Kong, Hong Kong.  \\
Email: \texttt{yulin2@cityu.edu.hk}

\item \hspace*{3.8mm}$\dag$ Department of Mathematics, City University of
Hong Kong, Hong Kong. \\
Email: \texttt{dandai@cityu.edu.hk}

\item \hspace*{3.8mm}$\ddag$ Department of Mathematics, City University of
Hong Kong, Hong Kong. \\
Email: \texttt{ptibboel@cityu.edu.hk} (corresponding author)

% added my e-mail address...

\end{description}

\newpage

%%%%%%%%%%%%%%%%%%%%%%%%%%%%%%%%%%%%%%%%%%%%%%%%%%%%%%%%%%%%%%%%%%%%%%%%%%%%%%%%%%

\section{Introduction and statement of results}

The third and fourth Painlev\'e equations are the following two second order nonlinear differential equations
\begin{eqnarray}
  \textrm{P}_{\textrm{III}}:  \qquad \frac{d^{2}u}{dx^{2}}&=&\frac{1}{u}\left(\frac{du}{dx}\right)^{2}-\frac{1}{x}\frac{du}{dx}+\frac{1}{x}\left(\alpha u^{2}+\beta\right)+\gamma u^{3}+\frac{\delta}{u}, \label{Painleve III} \\
  \textrm{P}_{\textrm{IV}}: \qquad \frac{d^{2}u}{dx^{2}}&=&\frac{1}{2u}\left(\frac{du}{dx}\right)^{2}+\frac{3}{2}u^{3}+4xu^{2}+2(x^{2}-\alpha)u+\frac{\beta}{u},  \label{Painleve IV}
\end{eqnarray}
where $\alpha, \beta , \gamma$ and $\delta$ are arbitrary constants. Their solutions, which are called the Painlev\'e transcendents, satisfy the famous Painlev\'e property, i.e. all of their movable singularities are poles. Although the Painlev\'e equations were originally studied from a purely mathematical point of view, it has been realised in recent years that they play an important role in many other mathematical and physical fields (see \cite{Conte-book,Fokas-book,Gromak-book} and references therein).
% Changed 'in the recent years' into 'in recent years' and changed ';see... therein' into '(see...therein)'.
Due to their importance as well as beautiful properties, the Painlev\'e transcendents are viewed as new members of the community of special functions (see \cite{Clark2006} and \cite[Chap. 32]{Olver-book}).
% Changed the '; see...' into '(see...)'.

An important part of studying Painlev\'e transcendents is to investigate their asymptotic behaviour. More specifically, we discern two directions of research: The first direction is to study asymptotics of certain solutions on the real axis and find the relations between their asymptotic behaviours at two endpoints of an interval, usually $(-\infty,\infty)$ or $(0,\infty)$. These relations are called connection formulas (see for example \cite{And:Kit,Bas:Cla:Law:McL,Dei:Zhou1995,Its:Kap1998}). Connection formulas have also been  studied on the imaginary axis (see \cite {Kit:Var2004,Kit:Var2010}).
The other direction is to study the asymptotics in the complex plane. Usually, Painlev\'e transcendents have poles in the complex plane, whose locations depend on the initial conditions of the corresponding differential equation.
% Added a comma after 'Usually' and added 'of the corresponding differential equation'.
But, in sectors containing one special ray in the complex $x$-plane,
 there also exist some solutions which are pole-free when $|x|$ is large.
 Such solutions were named \emph{tronqu\'ee} solutions by Boutroux \cite{boutroux1913} when he studied Painlev\'e equations one hundred years ago.
% Changed 'are called' into 'were named'.
 These solutions usually have free parameters appearing in exponentially small terms for large $|x|$ (see Its and Kapaev \cite{Its:Kap2003}).
 Moreover, there are some solutions without any parameters, which are pole-free in larger sectors containing three special rays. Such solutions are called \emph{tritronqu\'ee} solutions.
 Recently, Joshi and Kitaev \cite{Joshi2001} reconsidered these solutions and proved the existence and uniqueness of tritronqu\'ee solutions for the first Painlev\'e equation with more modern techniques.
 Later on, similar results were obtained for the first and second Painlev\'e hierarchies as well (see \cite{Dai:Zhang,Joshi2003-1,Joshi2003-2,Joshi2009}). Note that
 tronqu\'ee solutions for the fifth Painlev\' e equation are also studied by Andreev and Kitaev
\cite{And:Kit1997} and Shimomura \cite{Shimo}. In these two papers, the method used are different
from that in \cite{Joshi2001}.

In the literature, tronqu\'ee solutions to the first Painlev\'e equation
\begin{equation}
  \textrm{P}_{\textrm{I}}:  \qquad \frac{d^{2}u}{dx^{2}} = 6 u^2(x) -x
\end{equation}
are well studied.
For $\textrm{P}_{\textrm{I}}$, when studying the behaviour of their tritronqu\'ee solutions in the unbounded domain, people are also interested in their properties in the finite domain.
Based on numerical results (see for example \cite{Forn:Weid}), Dubrovin \cite{Dub:Gra:Kle} conjectures that the tritronqu\'ee solutions for $\textrm{P}_{\textrm{I}}$ are analytic in a neighbourhood of the origin and in a sector of central angle $8\pi/5$ containing the origin. This conjecture was proved by Costin, Huang and Tanveer \cite{Cos:Hua:Tan} very recently. It is not  clear whether similar conjectures are true for other Painlev\'e equations.
% Removed 'some' in 'Based on some numerical results.'

In this paper, we will follow Joshi and Kitaev's idea in \cite{Joshi2001} and study the tronqu\'ee solutions for the $\textrm{P}_{\textrm{III}}$ and $\textrm{P}_{\textrm{IV}}$ equations. Although some divergent power series solutions for these two equations are known, tronqu\'ee solutions, to the best of our knowledge, have not appeared in the literature. So we think it is worthwhile to prove the existence of these tronqu\'ee solutions, give the regions of validity of these solutions and their general asymptotic behaviour.
Moreover, it is worth mentioning that comparing with other methods like the isomonodromy method, Joshi and Kitaev's approach used in this paper
is more straightforward.
% Changed 'isomonodromy method' into 'the isomonodromy method'.
% I don't think we can say that Joshi and Kitaev's approach is easier to follow. I suggest we change 'elementary and easier to follow' into 'straightforward'.

Before we state our main results, we note that (see \cite[p.150]{Gromak-book} and \cite{Mil:Cla:Bas}), using B\"{a}cklund transformations, it suffices to study only two cases for the $\textrm{P}_{\textrm{III}}$ equations:
   (i) $\gamma \delta \neq 0$ and
   (ii) $\gamma \delta = 0$.
Moreover, these two cases can be reduced to the following canonical ones: (i) $\gamma = 1$ and $\delta = -1$ and (ii) $\alpha=1, \gamma=0$ and $\delta = -1$ without loss of generality (see \cite[p.150]{Gromak-book} and \cite{Mil:Cla:Bas}).
% Added '(see \cite[p.150]{Gromak-book} and \cite{Mil:Cla:Bas})' after 'we note that'.
Thus, for $\textrm{P}_{\textrm{III}}$,  we only need to study the following two equations:
\begin{align}\label{PIII Case 1}
  \textrm{P}_{\textrm{III}}^{(\textrm{i})}:  \qquad \frac{d^{2}u}{dx^{2}}=\frac{1}{u}\left(\frac{du}{dx}\right)^{2}-\frac{1}{x}\frac{du}{dx}+\frac{1}{x}\left(\alpha u^{2}+\beta\right)+u^{3}-\frac{1}{u},
\end{align}
and
\begin{align}\label{PIII Case 2}
  \textrm{P}_{\textrm{III}}^{(\textrm{ii})}:  \qquad \frac{d^{2}u}{dx^{2}}=\frac{1}{u}\left(\frac{du}{dx}\right)^{2}-\frac{1}{x}\frac{du}{dx}+\frac{1}{x}\left(u^{2}+\beta\right)-\frac{1}{u}.
\end{align}
For the above two equations, we have the following results.
\begin{thm} \label{thm-p3-1}
For $k=0,1$ and $x_0 \neq 0$, the following two statements hold:
% Changed 'we have' into 'the following two statements hold'.
  \begin{itemize}
    \item[1. ]  There exist one-parameter solutions of $\textrm{P}_{\textrm{III}}^{(\textrm{i})}$ in the sectors
    \begin{align}
      \hspace{-1cm} S^{(m)}_k = \begin{cases}
        \biggl\{x\in\mathbb{C}\; |\;|x|>|x_0|, - \frac{\pi}{2} + k \pi <\arg{x}< \frac{\pi}{2} + k \pi  \biggr\}, & \textrm{for }m=0,2, \vspace{2mm} \\
        \biggl\{x\in\mathbb{C}\; |\; |x|>|x_0|, k \pi <\arg{x}< (k+1) \pi  \biggr\}, & \textrm{for }m=1,3
      \end{cases}
    \end{align}
    with the following asymptotic expansion
    \begin{equation} \label{p3-1-expan}
      u(x)\sim\sum\limits_{n=0}^{\infty} a_{n}^{(m)}x^{-n} \qquad \textrm{as } |x| \to \infty.
    \end{equation}
    Here $a_{0}^{(m)} = e^{\frac{m\pi i}{2}}$, $m=0,1,2,3,$ and the subsequent coefficients $a_{n+1}^{(m)}$, $n \geq 0$, are determined by the recurrence relation
    \begin{equation} \label{p3-1-recur}
      \begin{cases}
      2a_{0}^{(m)}A_{0}^{(m)}  a_{n+1}^{(m)}=(\beta-1-n)a_{n}^{(m)}- A_{0}^{(m)}\sum\limits_{l=1}^{n}a_{l}^{(m)}a_{n+1-l}^{(m)}\\
      \hspace{3.1cm} -\sum\limits_{k=1}^{n+1}A_k^{(m)} \sum\limits_{l=0}^{n+1-k}a_{l}^{(m)}a_{n+1-k-l}^{(m)}  \\
      2a_{0}^{(m)}A_{0}^{(m)}A_{n+1}^{(m)}=\alpha \delta_{n,0} + (n+\beta-2)A_{n}^{(m)} - a_{0}^{(m)}\sum\limits_{l=1}^{n}A_{l}^{(m)}A_{n+1-l}^{(m)}\\
      \hspace{3.1cm} - \sum\limits_{k=1}^{n}a_{k}^{(m)}\sum\limits_{l=0}^{n+1-k}A_{l}^{(m)}A_{n+1-k-l}^{(m)}
    \end{cases}
    \end{equation}
    with $A_0^{(m)}= - (a_0^{(m)})^2$, $\delta_{0,0} =1$ and $\delta_{n,0} =0$ for $n>0.$

    \item[2. ] There exists a unique solution of $\textrm{P}_{\textrm{III}}^{(\textrm{i})}$ in each of the following sectors
    \begin{align} \label{p3-1-large-sector}
      \hspace{-1cm} \Omega^{(m)}_k = \begin{cases}
        \biggl\{x\in\mathbb{C}\; |\;|x|>|x_0|, - \frac{\pi}{2} + k \pi <\arg{x}< \frac{3\pi}{2} + k \pi  \biggr\}, & \textrm{for }m=0,2, \vspace{2mm} \\
        \biggl\{x\in\mathbb{C}\; |\; |x|>|x_0|, k \pi <\arg{x}< (k+2) \pi  \biggr\}, & \textrm{for }m=1,3
      \end{cases}
    \end{align}
    with the asymptotic expansion given in \eqref{p3-1-expan}.
  \end{itemize}
\end{thm}
\begin{rmk}
  It is known that there exist rational solutions for $\textrm{P}_{\textrm{III}}^{(\textrm{i})}$ if and only if $\alpha \pm \beta = 4k, k \in \mathbb{Z}$ (see \cite{Clark2006,Mil:Cla:Bas}). As these rational solutions satisfy the above asymptotic expansion \eqref{p3-1-expan}, they are the unique tronqu\'ee solutions in the second part of Theorem~\ref{thm-p3-1}.
  They are pole-free in $\Omega^{(m)}_k$ when $|x|$ is large since all of their poles are located in a bounded region in the complex $x$-plane. Moreover, our theorem suggests that, even for arbitrary parameters $\alpha$ and $\beta$, there still exist tronqu\'ee solutions which are pole-free in a pretty big sector $\Omega^{(m)}_k$ for large $x$. We think this is an interesting observation.
\end{rmk}

\begin{rmk} \label{rmk-anAn}
  The readers may wonder why there is another set of coefficients $\{A_n^{(m)}\}$ in the recurrence relations \eqref{p3-1-recur}. The reason is that, instead of considering the $\textrm{P}_{\textrm{III}}^{(\textrm{i})}$ equation \eqref{PIII Case 1} directly, we consider an equivalent system of first order differential equations instead. Similar situations occur in the  $\textrm{P}_{\textrm{III}}^{(\textrm{ii})}$ and $\textrm{P}_{\textrm{IV}}$ cases as well. The detailed explanations will be provided later in this section.
\end{rmk}

\begin{thm} \label{thm-p3-2}
  For $k=0,1,2,3$ and $x_0 \neq 0$, the following two statements hold true:
  % Changed 'we have' into 'the following two statements hold true'.
  \begin{itemize}
    \item[1. ]  There exist one-parameter solutions of $\textrm{P}_{\textrm{III}}^{(\textrm{ii})}$ in the sectors
    \begin{align}
      \hspace{-1.05cm} S^{(m)}_k = \begin{cases}
        \biggl\{x\in\mathbb{C}\; |\;|x|>|x_0|, - \frac{3\pi}{4} - \frac{k \pi}{2} <\arg{x}< \frac{3\pi}{4} - \frac{k \pi}{2}  \biggr\}, & \textrm{for }m=0,2, \vspace{2mm} \\
        \biggl\{x\in\mathbb{C}\; |\; |x|>|x_0|, \frac{\pi}{4} - \frac{k \pi}{2} <\arg{x}< \frac{7\pi}{4} - \frac{k \pi}{2}  \biggr\}, & \textrm{for }m=1
      \end{cases}
    \end{align}
    with the following asymptotic expansion
    \begin{equation} \label{p3-2-expan}
      u(x)\sim x^{\frac{1}{3}}\sum\limits_{n=0}^{\infty} a_{n}^{(m)}x^{-\frac{2n}{3}} \qquad \textrm{as } |x| \to \infty.
    \end{equation}
    Here $a_{0}^{(m)} = e^{\frac{2m\pi i}{3}}$, $m=0,1,2,$ and the subsequent coefficients $a_{n+1}^{(m)}$, $n \geq 0$, are determined by the recurrence relation
    \begin{equation} \label{p3-2-recur}
      \begin{cases}
      [a_0^{(m)}]^{2}(6a_{n+1}^{(m)} - 3 A_{n+1}^{(m)}) = (2n+2-3\beta)a_{n}^{(m)}
      -3a_0^{(m)}\sum\limits_{l=1}^{n}a_{l}^{(m)}a_{n+1-l}^{(m)}\\
      \hspace{4cm} + 3\sum\limits_{k=1}^{n}A_{k}^{(m)}\sum\limits_{l=0}^{n+1-k}a_{l}^{(m)}a_{n+1-k-l}^{(m)}\\
      [a_0^{(m)}]^{2}(3a_{n+1}^{(m)}-6A_{n+1}^{(m)})=(3\beta-4+2n)A_{n}^{(m)}
      -3a_0^{(m)}\sum\limits_{l=1}^{n}A_{l}^{(m)}A_{n+1-l}^{(m)}\\
      \hspace{4cm} -3\sum\limits_{k=1}^{n}a_{k}^{(m)}\sum\limits_{l=0}^{n+1-k}A_{l}^{(m)}A_{n+1-k-l}^{(m)}
    \end{cases}
    \end{equation}
    with $A_0^{(m)}= - a_0^{(m)}$.

    \item[2. ] For any branch of $x^{1/3},$ there exists a unique solution of $\textrm{P}_{\textrm{III}}^{(\textrm{ii})}$ in $\mathbb{C}\setminus \Gamma$
    with the asymptotic expansion given in \eqref{p3-2-expan}. Here $\Gamma$ is an arbitrary branch cut connecting 0 and $\infty$.
  \end{itemize}
\end{thm}

\begin{rmk}
  In the literature,  \eqref{PIII Case 2} is also called the $\textrm{P}_{\textrm{III}}$ of type $D_7$
  according to the algebro-geometric classification scheme in \cite{Ohy:Kaw:Sak:Oka}. It is also called the degenerate $\textrm{P}_{\textrm{III}}$
  by Kitaev and Vartanian in \cite{Kit:Var2004,Kit:Var2010}.
  % Removed 'the $\textrm{P}_{\textrm{III}}^{(\textrm{ii})}$ equation' after 'In the literature.'
\end{rmk}

\begin{rmk} \label{rmk-kv1}
  Kitaev and Vartanian studied the connection formulas for the $\textrm{P}_{\textrm{III}}^{(\textrm{ii})}$ equation in \cite{Kit:Var2004,Kit:Var2010}. By using the isomonodromy deformation method, they obtained asymptotics for solutions $u(x)$ when $x \to \pm \infty$ and $x \to \pm i \infty$.
  % Removed '$x$ belongs to the real and imaginary axes and tends to infinity, that is,' after $u(x)$.
  These solutions are similar to the one-parameter solutions in the first part of our Theorem~\ref{thm-p3-2}. Their asymptotic formulas involve a term $x^{1/3}$ and a trigonometric term.
  See for example formulas (39), (41) and (42) in \cite{Kit:Var2004}.
  % Changed ', for example see' into '. See for example'.
\end{rmk}

\begin{rmk}
  In \cite{Kit:Var2010}, Kitaev and Vartanian obtained significant asymptotic formulas for $u(x)$ which are valid in the neighbourhood of poles.
  % Removed 'another' after 'obtained.'
  In addition, they got asymptotics for these poles, see Theorem~2.1-2.3 in \cite{Kit:Var2010}. Note that these poles are located on the real or imaginary axis depending on different monodromy data.
  %This agrees with our results about the validity sectors in the first part of Theorem \ref{thm-p3-2}.
\end{rmk}

We obtain the following results for $\textrm{P}_{\textrm{IV}}$:
% Changed 'Similarly, we' into 'We' and '$\textrm{P}_{\textrm{IV}}$.' into '$\textrm{P}_{\textrm{IV}}$:'
   \begin{thm}\label{p4-thm}
  For $k=0,1,2,3$ and $x_0 \neq 0$, the following statements hold true:
  % Changed 'we have' into 'the following statements hold true:'
  \begin{itemize}
    \item[1. ] There exist one-parameter solutions of $\textnormal{P}_{\textnormal{IV}}$ in the regions
        % Changed 'respective regions' into 'regions'.
\[
 S_{k}^{(m)} =\left\{
\begin{aligned}
          &\left\{x\in\mathbb{C} \big | |x|>|x_0|, \frac{1}{2}k\pi<\arg{x}<\frac{1}{2}\pi+\frac{1}{2}k\pi\right\}, && \mbox{for } m=1,
          \\
         &\left\{x\in\mathbb{C}\big | |x|>|x_0|,-\frac{1}{4}\pi+\frac{1}{2}k\pi<\arg{x}<\frac{1}{4}\pi+\frac{1}{2}k\pi\right\},
         && \mbox{for } m=2,3,4,
        \end{aligned}
        \right.
\]
with the following asymptotic expansion
\begin{equation}\label{p4-expan}
u(x)\sim  \left\{
\begin{aligned}
	&x\sum\limits_{n=0}^{\infty}a_{n}^{(m)}x^{-2n},
	 && \mbox{for } m=1,2,
	\\
	&\frac1x \sum\limits_{n=0}^{\infty}a_{n}^{(m)}x^{-2n},
	 && \mbox{for } m=3,4,
\end{aligned}
\right.
\qquad \mbox{as } |x|\to \infty,
\end{equation}
Here $a_0^{(1)}=-\frac23$, $a_0^{(2)}=-2$, $a_0^{(3)}=\k$, $a_0^{(4)}=-\k$ with $\a=-\k+2\K+1$ and
$\b=-2\k^2$, and
the subsequent coefficients $a_n^{(m)}$ can be determined by the following recurrence relations:
if $m=1$,
  \begin{equation}\label{p4-recurrence-1}
	\begin{cases}
      \frac{2}{3}a_{n+1}^{(1)}-\frac{8}{3}A_{n+1}^{(1)}=(1-2n)a_{n}^{(1)}-\sum\limits_{k=1}^{n}a_{k}^{(1)}(4A_{n+1-k}^{(1)}-a_{n+1-k}^{(1)}), \\
      \frac{2}{3}A_{n+1}^{(1)}-\frac{2}{3}a_{n+1}^{(1)}=(2n-1)A_{n}^{(1)}+2\sum\limits_{k=1}^{n}A_{k}^{(1)}(a_{n+1-k}^{(1)}-A_{n+1-k}^{(1)})
    \end{cases}
\end{equation}
with $a_{1}^{(1)}=\a$, $A_{0}^{(1)}=\frac13$, $A_{1}^{(1)}=\frac12-\k+\frac12 \K$;
if $m=2$,
  \begin{equation}\label{p4-recurrence-2}
\begin{cases}
      a_{n+1}^{(2)}=(\frac12-n)a_{n}^{(2)}+4 A_n+\frac12\sum\limits_{k=1}^{n}a_{k}^{(2)}
      (a_{n+1-k}^{(2)}- 4A_{n-k}^{(2)}), \quad n\geq 1, \\
      A_{n+1}^{(2)}=(\frac12+n)A_{n}^{(2)}+\sum\limits_{k=0}^{n}A_{k}^{(2)}
      (a_{n+1-k}^{(2)}-A_{n-k}^{(2)}), \qquad \qquad \  n\geq 0
    \end{cases}
\end{equation}
with $a_{1}^{(2)}=-\a$, $A_{0}^{(2)}=-\frac12\K$; if $m=3$,
  \begin{equation}\label{p4-recurrence-3}
	\begin{cases}
          a_{n+1}^{(3)}=-(n+\frac12)a_{n}^{(3)}+\frac12\sum\limits_{k=0}^{n}a_{k}^{(3)}(a_{n-k}^{(3)}
          -4A_{n+1-k}^{(3)}), \qquad \   n\geq 0, \\
     A_{n+1}^{(3)}=(n-\frac12)A_{n}^{(3)}+a_n^{(3)}-\sum\limits_{k=1}^{n}A_{k}^{(3)}
     (A_{n+1-k}^{(3)}-a_{n-k}^{(3)}), \quad\  n\geq 1
    \end{cases}
\end{equation}
with  $A_{0}^{(3)}=1$, $A_{1}^{(3)}=-\frac12 (1-2\k+\K)$; and if $m=4$,
  \begin{equation}\label{p4-recurrence-4}
	\begin{cases}
       a_{n+1}^{(4)}=(n+\frac12)a_{n}^{(4)}-\frac12\sum\limits_{k=0}^{n}a_{k}^{(4)}(a_{n-k}^{(4)}-4A_{n-k}^{(4)}), \\
     A_{n+1}^{(4)}=-(n+\frac12)A_{n}^{(4)}+\sum\limits_{k=0}^{n}A_{k}^{(4)}(A_{n-k}^{(4)}-a_{n-k}^{(4)})
    \end{cases}
\end{equation}
with  $A_{0}^{(4)}=\frac12\K$.

    \item[2. ] There exist a unique solution
    of $\textnormal{P}_{\textnormal{IV}}$ whose asymptotic behaviour as $|x|\to \infty$ is given by
   \eqref{p4-expan} respectively in the regions
\[
  \Omega_{k}^{(m)} =\left\{
\begin{aligned}
         &\Big\{x\in\mathbb{C}\big | |x|>|x_0|, k\pi<\arg{x}<(k+2)\pi\Big\},
          && \mbox{for } m=1,
         \\
        &\left\{x\in\mathbb{C}\big | |x|>|x_0|, -\frac{3}{4}\pi+k\pi<\arg{x}<\frac{5}{4}\pi+k\pi\right\},
        && \mbox{for } m=2,3,4.
        \end{aligned}
\right.
\]
    \end{itemize}
\end{thm}

\begin{rmk}
  Its and Kapaev applied isomonodromy and Riemann-Hilbert methods to obtain the asymptotics for the Clarkson-McLeod solution solutions $u(x)$ as $x \to - \infty$.
  % Removed 'Like in Remark \ref{rmk-kv1}, '.
  This solution is also similar to the one-parameter solutions in the first part of our Theorem \ref{p4-thm}. The asymptotic formula involves a term $-\frac{2}{3}x$ and a cosine term, see \cite[eq. (1.10)]{Its:Kap1998}.
\end{rmk}

To prove the above theorems, we need the following theorem by Wasow (see \cite[Theorem
12.1]{Wasow}).
% Changed '\cite[Theorem 12.1]{Wasow}' into '(see \cite[Theorem 12.1]{Wasow})'.

\begin{thm}\label{Wasow Thm}

Let $S$ be an open sector of the complex $x$-plane with vertex at
the origin and a positive central angle not exceeding $\pi/(q+1)$
$($q a nonnegative integer$)$. Let $f(x,w)$ be an $N$-dimensional
vector function of $x$ and an $N$-dimensional vector $w$ with the
following properties:
% Changed 'the $N$-dimensional vector $w$' into 'an $N$-dimensional vector $w$' and changed 'properties.' into 'properties:'
\begin{enumerate}

\item $f(x,w)$ is a polynomial in the components $w_j$ of $w$,
$j=1,\cdots,N$, with coefficients that are holomorphic in $x$ in the
region $
0<x_0\leq|x|<\infty,$ $x \in S$, where $x_0$ is a constant.

\item The coefficients of the polynomial $f(x,w)$ have asymptotic
series in powers of $x^{-1}$, as $x\rightarrow\infty$, in $S$.

\item If $f_{j}(x,w)$ denotes the components of $f(x,w)$ then all
the eigenvalues $\lambda_{j}$, $j=1,2,\cdots,N$ of the Jacobian
matrix
\begin{equation}
\Big\{ \lim_{x\rightarrow\infty, x\in S} \Big(\frac{\partial
f_{j}}{\partial w_{k}} \Big\vert_{w=0} \Big)\Big\}
\end{equation}
are different from zero.

\item The differential equation
\begin{equation}\label{wasow model eqn}
x^{-q} w'=f(x,w)
\end{equation}
is formally satisfied by a power series of the form $ \D\sum_{n=1}^{\infty}a_n x^{-n}.$

\end{enumerate}

Then there exists, for sufficiently large $x$ in $S$, a solution
$w=\phi(x)$ of \eqref{wasow model eqn} such that, in every proper
subsector of $S$,
\begin{equation}
\phi(x)\sim\sum_{n=1}^{\infty} a_n x^{-n}, \qquad x\to\infty.
\end{equation}
\end{thm}

Note that in Wasow's theorem, the function $f(x,w)$ in \eqref{wasow model eqn} needs to be  polynomial in terms of the components of the unknown function $w$.
% Changed 'needs to be a polynomial in the unknown function $w$' into 'needs to be  polynomial in terms of the components of the unknown function $w$'.
However, there are $\frac{1}{u}$ terms in both $\textrm{P}_{\textrm{III}}$ and $\textrm{P}_{\textrm{IV}}$ equations in \eqref{Painleve III} and \eqref{Painleve IV}. This means we cannot directly apply Wasow's theorem like the $\textrm{P}_{\textrm{I}}$ and $\textrm{P}_{\textrm{II}}$ cases in \cite{Joshi2001,Joshi2003-2}.
% Changed 'can not' into 'cannot'.
 Fortunately, there exist first order differential systems for Painlev\'e equations where all unknown functions are given in polynomial terms.
 % Changed 'we notice that there are some' into 'there exist'.
 % Removed 'in the literature'.
 We will consequently use these first order differential systems to derive our results.
 % Changed 'Then, we can start from' into 'We will consequently use'.
 This also explains Remark \ref{rmk-anAn}, where $\{A_n^{(m)}\}$ are actually coefficients for the other solution in the systems.

In the remaining part of this paper, we will prove Theorems 1-3 in Sections 2-4, respectively.
% Changed 'prove our' into 'prove'.
Because Sections 2-4 are self-contained,
we shall use the same notation for the functions and variables but it will mean different things in different sections. We hope
that this will not cause any confusion.

%%%%%%%%%%%%%%%%%%%%%%%%%%%%%%%%%%%%%%%%%%%%%%%%%%%%%%%%%%%%%%%%%%%%%%%%%%%%%%%%%%%%%%%%

\section{The $\textrm{P}_{\textrm{III}}^{(\textrm{i})}$ equations} \label{sec-p3-1}

The $\textrm{P}_{\textrm{III}}$ equations \eqref{Painleve III} can be written as the following system of first order differential equations:
% Changed 'first order differential system' into 'system of first order differential equations:'.
\begin{align}\label{p3-first-order}
\begin{cases}
  x\D\frac{du}{dx} = (-\delta)^{1/2} x+ h u+xu^{2}U,  \vspace{2mm} \\
  x\D\frac{dU}{dx}=\alpha+\gamma x u-(1+h)U-xuU^{2},
  \end{cases}
\end{align}
where $h = 1 - \beta (-\delta)^{-1/2}$ (see \cite[p.149]{Gromak-book}). That is, eliminating $U$ from the above equations gives us \eqref{Painleve III}.
% Changed '; see \cite[p.149]{Gromak-book}' into '(see \cite[p.149]{Gromak-book})'.
For case (i), $\gamma =1$, $\delta = -1$ and the first order equations are
% Changed 'For case (i) $\gamma =1$ and $\delta = -1$, the first order equations are' into 'For case (i), $\gamma =1$ and $\delta = -1$ and the first order equations are'.
\begin{align}\label{p3-1-first-order}
\begin{cases}
  x\D\frac{du}{dx} =  x+ (1-\beta) u+xu^{2}U,  \vspace{2mm} \\
  x\D\frac{dU}{dx}=\alpha+  x u-(2-\beta)U-xuU^{2}.
  \end{cases}
\end{align}
Then we get the formal solutions to the above system.

\begin{prop} \label{prop-p3-1-formal}
  We have the following formal solutions for \eqref{p3-1-first-order}
  \begin{equation} \label{p3-1-formal}
    u_{f}^{(m)}(x)=\sum\limits_{n=0}^\infty a_{n}^{(m)}x^{-n} \quad \textrm{and} \quad U_{f}^{(m)}(x)=\sum\limits_{n=0}^\infty A_{n}^{(m)}x^{-n},
  \end{equation}
  where
  \begin{equation} \label{p3-1-initial}
    a_0^{(m)} = e^{\frac{m\pi i}{2}} \quad \textrm{and} \quad A_0^{(m)}= - (a_0^{(m)})^2, \qquad m=0,1,2,3,
  \end{equation}
  and the subsequent coefficients are determined by \eqref{p3-1-recur}.
\end{prop}
\begin{proof}
  Inserting \eqref{p3-1-formal} into \eqref{p3-1-first-order} gives us the results.
  % Changed 'Substituting' into 'Inserting'.
\end{proof}
\begin{rmk}
  % Changed 'The first-order differential systems' into 'The systems of first order differential equations'.
  The systems of first order differential equations for Painlev\'e equations are not unique.
  For $\textrm{P}_{\textrm{III}}^{(\textrm{i})}$, besides \eqref{p3-1-first-order}, one can get the same results as in the above Proposition by studying the following system obtained from the Hamiltonian $H_{III}$
 \begin{align}\label{hiii}
  \begin{cases}
  x\D\frac{du}{dx}=x\frac{\partial H_{III}}{\partial U}=4u^{2}U-xu^{2}+(1-\beta)u+x,  \vspace{2mm} \\
  x\D\frac{dU}{dx}=-x\frac{\partial H_{III}}{\partial u}=-4uU^{2}+(2ux+\beta-1)U+\frac{1}{4}(2+\alpha-\beta)x,
  \end{cases}
 \end{align}
 where
  \begin{align}
    xH_{III}(x,u,U):=2u^{2}U^{2}-xu^{2}U+(1-\beta)uU+xU-\frac{1}{4}(2+\alpha-\beta)xu.
  \end{align}
  Since \eqref{hiii} is only valid for case (i), where $\gamma = 1$ and $\delta = -1$, it is more convenient to use \eqref{p3-first-order} to consider both cases in the current and next sections.
  % Changed 'only valid' into 'is only valid'.
  % Do we write 'Case' or 'case'?
  % Changed 'Case (i) $\gamma = 1$ and $\delta = -1$' into 'case (i), where $\gamma = 1$ and $\delta = -1$'
\end{rmk}

Although the original $\textrm{P}_{\textrm{III}}^{(\textrm{i})}$ equations \eqref{PIII Case 1} have $\frac{1}{u}$ terms, the right-hand side of the first order system \eqref{p3-1-first-order} is a polynomial in the functions $u$ and $U$.
% Changed 'unknown functions' into 'functions'.
Thus Wasow's theorem \ref{Wasow Thm} is applicable and we have the following results.
% Changed 'Then' into 'Thus'.

\begin{prop}
  In any sector of angle less than $\pi$, there exists a solution $u(x)$ of $\textrm{P}_{\textrm{III}}^{(\textrm{i})}$ whose asymptotic behaviour as $|x| \to \infty$ is given by $u_{f}^{(m)}(x)$ in \eqref{p3-1-formal} with $m=0,1,2,3$.
\end{prop}
\begin{proof}
  To arrive at the standard form \eqref{wasow model eqn} in Wasow's theorem, we introduce $v(x)=u(x)-a_0^{(m)}$ and $V(x)= U(x)-A_0^{(m)}$. Then $(v,V)$ satisfies the following system
  \begin{equation} \label{p3-1-first-order-new}
    \frac{d}{dx} \left(\begin{matrix} v \\ V \end{matrix}\right) = \left( \begin{matrix}
      1+ \frac{1-\beta}{x}(v+a_0^{(m)})+ (v+a_0^{(m)})^{2}(V+A_0^{(m)}) \\ \frac{\alpha}{x} + v+a_0^{(m)}-\frac{2-\beta}{x}(V+A_0^{(m)}) - (v+a_0^{(m)})(V+A_0^{(m)})^{2}
    \end{matrix} \right)
  \end{equation}
  and has a formal expansion given by
  \begin{equation*}
    v_{f}^{(m)}(x)=\sum\limits_{n=1}^\infty a_{n}^{(m)}x^{-n} \quad \textrm{and} \quad V_{f}^{(m)}(x)=\sum\limits_{n=1}^\infty A_{n}^{(m)}x^{-n}.
  \end{equation*}
   It is easy to see that conditions 1, 2 and 4 in Wasow's theorem are satisfied with $q=0$.
   % Changed 'the conditions' into 'conditions'.
   To verify condition 3, let us denote the right-hand side of \eqref{p3-1-first-order-new} by $(f_1, f_2)^T$.
   % Changed 'the condition' into 'condition'.
   Then we have
  \begin{align} \label{p3-1-J}
    \lim\limits_{x\rightarrow\infty}\left.\begin{pmatrix}
      \frac{\partial f_{1}}{\partial v} & \frac{\partial f_{1}}{\partial V} \\
      \frac{\partial f_{2}}{\partial v} & \frac{\partial f_{2}}{\partial V}
    \end{pmatrix}\right|_{v=V=0}=\begin{pmatrix}
      2 a_0^{(m)} A_0^{(m)} & (a_0^{(m)})^{2}\\
      0 & -2a_0^{(m)}A_0^{(m)}
    \end{pmatrix}:= J.
  \end{align}
  The eigenvalues of $J$ are
  \begin{equation} \label{p3-1-J-eig}
    \lambda_{1,2} = \pm 2 a_0^{(m)} A_0^{(m)} = \mp 2 e^{-\frac{m \pi i}{2}} \neq 0, \qquad m = 0,1,2,3;
  \end{equation}
  (see \eqref{p3-1-initial} for the values of $a_0^{(m)}$ and $ A_0^{(m)}$).
  % Changed 'see \eqref{p3-1-initial} for the values of $a_0^{(m)}$ and $ A_0^{(m)}$' into '(see \eqref{p3-1-initial} for the values of $a_0^{(m)}$ and $ A_0^{(m)}$)'
  Thus our proposition follows from Wasow's theorem.
  % Changed 'Then' into 'Thus'.
\end{proof}

The above proposition gives us the existence of solutions to $\textrm{P}_{\textrm{III}}^{(\textrm{i})}$ with specific asymptotic expansions given in \eqref{p3-1-expan}. However, these solutions may not be unique because some exponentially small terms with arbitrary parameters could exist. To understand these solutions better, more detailed analysis about their properties is needed. This will prove our Theorem \ref{thm-p3-1}.

\bigskip

\noindent\emph{Proof of Theorem \ref{thm-p3-1}.} Let $(u_0, U_0)$ be a solution of \eqref{p3-1-first-order} with asymptotic behaviour as given in \eqref{p3-1-formal}.
% Changed 'given' into 'as given'
Consider a perturbation
\begin{equation}
  u=u_0 + \widehat u \quad \textrm{and} \quad U=U_0 + \widehat U, \qquad |\widehat u|, |\widehat U| \ll 1
\end{equation}
of this solution.
% Moved 'of this solution' to the end of the sentence, moved the full stop accordingly and changed '\widehat u, \widehat U' into '|\widehat u|, |\widehat U|'.
Then $\widehat u$ and  $\widehat U$ satisfy
% Removed 'the equation below' after 'satisfy'.
\begin{equation*}
  x\frac{d}{dx} \left(\begin{matrix} \widehat u \\ \widehat U \end{matrix}\right) = \left( \begin{matrix} (1-\beta+2xu_0 U_0) \widehat u + x u_{0}^{2} \widehat U  + 2x u_0 \widehat u \widehat U + xU_0 \widehat u^2 + x \widehat u^2 \widehat U \\
      x(1-U_{0}^{2}) \widehat u + (\beta-2-2xu_{0}U_{0}) \widehat U - 2x U_0 \widehat u \widehat U - x u_0 \widehat U^2 - x \widehat u\widehat U^2
    \end{matrix} \right).
\end{equation*}
Since $|\widehat u|, |\widehat U| \ll 1$, it is sufficient to consider the following linear system to determine the asymptotic behaviour of $\widehat u$ and $\widehat U$
\begin{equation}
  \frac{d}{dx} \left(\begin{matrix} \widehat u \\ \widehat U \end{matrix}\right) =\begin{pmatrix}
      \frac{1-\beta}{x}+2u_0 U_0 & u_{0}^{2} \\
      1-U_{0}^{2} & \frac{\beta-2}{x}-2u_{0}U_{0}
    \end{pmatrix} \left(\begin{matrix} \widehat u \\ \widehat U \end{matrix}\right):= B(x) \left(\begin{matrix} \widehat u \\ \widehat U \end{matrix}\right).
\end{equation}
By \eqref{p3-1-formal}, the coefficient matrix $B(x)$ in the above equation has the following asymptotic behaviour:
% Changed 'have' into 'has' and 'behavior' into 'behaviour:'.
\begin{equation} \label{p3-1-B}
B(x) = J +\mathcal{O}\left(\frac{1}{x}\right), \qquad \textrm{as } |x| \to \infty,
\end{equation}
where $J$ is the matrix in \eqref{p3-1-J} with eigenvalues given in \eqref{p3-1-J-eig}.
% Is this explanation necessary?
So, by Theorem 12.3 in \cite{Wasow},
% Changed 'Then, according to' into 'So, by'.
% Removed 'the fundamental solution of the above equation is given by'.
\begin{equation}
  Y(x) x^D \exp\biggl(x \, \textrm{diag}(\lambda_1, \lambda_2)\biggr),
\end{equation}
where $D$ is a constant diagonal matrix $D= \textrm{diag}(d_1, d_2)$
and $Y(x)$ has an asymptotic expansion
\begin{equation*}
  Y(x) \sim \sum_{r=0}^\infty Y_r x^{-r}, \quad \det Y_0 \neq 0,  \qquad \textrm{as } |x| \to \infty.
\end{equation*}
Consequently, we have
% Changed 'Then' into 'Consequently'.
\begin{equation}
  \widehat u(x) \sim c_1 x^{d_1} e^{\lambda_1 x } + c_2  x^{d_2} e^{\lambda_2 x }, \qquad \textrm{as } |x| \to \infty
\end{equation}
with two arbitrary parameters $c_1$ and $c_2$. To ensure our assumption $|\widehat u| \ll 1$,
% Changed '\widehat u' into '|\widehat u|'.
we need to choose either $c_1 = 0$, or $c_2 = 0$,
% Changed '$c_i = 0$' into '$c_1 = 0$, or $c_2 = 0$'.
or a correct sector such that $\re e^{\lambda_i x } < 0$.
According to the specific values of $\lambda_{1,2}$ in \eqref{p3-1-J-eig}, we prove the first part of our theorem.

To prove the second part of our theorem and achieve uniqueness, we consider solutions in the following sectors
% Changed 'the uniqueness' into 'uniqueness'.
\begin{align*}
      S^{(m)}_{k, \pm\varepsilon} = \begin{cases}
        \biggl\{x\in\mathbb{C}\; |\;|x|>|x_0|, - \frac{\pi}{2} + k \pi \pm\varepsilon <\arg{x}< \frac{\pi}{2} + k \pi \pm\varepsilon \biggr\}, & \textrm{for }m=0,2, \vspace{2mm} \\
        \biggl\{x\in\mathbb{C}\; |\; |x|>|x_0|, k \pi\pm\varepsilon <\arg{x}< (k+1) \pi \pm\varepsilon \biggr\}, & \textrm{for }m=1,3
      \end{cases}
\end{align*}
with a small $\varepsilon>0$.
Let $(u_1, U_1)$ and $(u_2, U_2)$ be two solutions of \eqref{p3-1-first-order} whose asymptotic expansions are given by \eqref{p3-1-formal} in sectors $S^{(m)}_{k, \varepsilon}$ and $S^{(m)}_{k+1, -\varepsilon}$, respectively. Consider the difference of these solutions
\begin{equation}
  w:= u_1-u_2 \quad \textrm{and} \quad W:=U_1-U_2,
\end{equation}
which we define in the overlap of the two sectors
% Changed 'are defined' into 'we define'.
\begin{equation}
  \widehat S^{(m)}_{k,\varepsilon} = S^{(m)}_{k, \varepsilon} \cap S^{(m)}_{k+1, -\varepsilon}.
\end{equation}
Since $(u_1, U_1)$ and $(u_2, U_2)$ satisfy the same asymptotic expansion, we have
% Shouldn't these solutions be in parentheses?
\begin{equation} \label{p3-1-w-asy}
  w(x) = O(x^{-j}) \quad \textrm{and} \quad W(x) = O(x^{-j}) \qquad \textrm{as } |x| \to \infty \textrm{ and } x \in \widehat S^{(m)}_{k,\varepsilon}
\end{equation}
for all $j \in \mathbb{N}$. Moreover, one can verify using \eqref{p3-1-first-order} that $w(x)$ and $W(x)$ satisfy the following equations
% Changed 'verify' into 'verify using \eqref{p3-1-first-order}'
\begin{equation}
  \frac{d}{dx} \left(\begin{matrix} w \\  W \end{matrix}\right) =\begin{pmatrix}
      \frac{1-\beta}{x}+ (u_1+u_2) U_1 & u_{2}^{2} \\
      1-U_{2}^{2} & \frac{\beta-2}{x}-u_{1} (U_{1}+U_2)
    \end{pmatrix} \left(\begin{matrix} w \\ W \end{matrix}\right):= \widetilde B(x) \left(\begin{matrix} w \\ W \end{matrix}\right).
\end{equation}
The coefficient matrix $\widetilde B(x)$ satisfies the same asymptotic behaviour as $B(x)$ does in \eqref{p3-1-B}. Then by the same argument as in the first part,
% Changed 'similar analysis' into 'same argument'.
we get
% Changed 'have' into 'get'.
\begin{equation}
  w(x) \sim c_1 x^{d_1} e^{\lambda_1 x } + c_2  x^{d_2} e^{\lambda_2 x }, \qquad \textrm{as } |x| \to \infty.
\end{equation}
To ensure $w(x)$ satisfies the asymptotic behaviour in \eqref{p3-1-w-asy} for $x \in \widehat S^{(m)}_{k,\varepsilon}$, we must choose $c_1=c_2 =0$ in the above formula. This means $u_1 = u_2$ in $\widehat S^{(m)}_{k,\varepsilon}$ and the sector of validity
can be extended to $ \Omega^{(m)}_k$ in \eqref{p3-1-large-sector}. This completes the proof of our results. \hfill $\Box$

\bigskip

%%%%%%%%%%%%%%%%%%%%%%%%%%%%%%%%%%%%%%%%%%%%%%%%%%%%%%%%%%%%%%%%%%%%%%%%%%%%%%%%%%%%%

\section{The $\textrm{P}_{\textrm{III}}^{(\textrm{ii})}$ equations}

For case (ii), where $\alpha=1, \gamma=0$ and $\delta = -1$, the first order equations \eqref{p3-first-order} reduce to the following form:
% Changed 'case (ii) $\alpha=1, \gamma=0$ and $\delta = -1$' into 'case (ii), where $\alpha=1, \gamma=0$ and $\delta = -1$'
% Changed 'form' into 'form:'.
\begin{align}\label{p3-2-first-order}
\begin{cases}
  x\D\frac{du}{dx} =  x+ (1-\beta) u+xu^{2}U,  \vspace{2mm} \\
  x\D\frac{dU}{dx}= 1 -  (2-\beta)U-xuU^{2}.
  \end{cases}
\end{align}
Then we obtain the formal solutions.

\begin{prop} \label{prop-p3-2-formal}
  We have the following formal solutions for \eqref{p3-2-first-order}
  \begin{equation} \label{p3-2-formal}
    u_{f}^{(m)}(x)=x^{\frac{1}{3}}\sum\limits_{n=0}^{\infty} a_{n}^{(m)}x^{-\frac{2n}{3}} \quad \textrm{and} \quad U_{f}^{(m)}(x)=x^{-\frac{2}{3}}\sum\limits_{n=0}^{\infty}A_{n}^{(m)}x^{-\frac{2n}{3}},
  \end{equation}
  where
  \begin{equation} \label{p3-2-initial}
    a_0^{(m)} = e^{\frac{2m\pi i}{3}} \quad \textrm{and} \quad A_0^{(m)}= - a_0^{(m)}, \qquad m=0,1,2,
  \end{equation}
  and the subsequent coefficients are determined by \eqref{p3-2-recur}.
\end{prop}
\begin{proof}
  Inserting \eqref{p3-2-formal} into \eqref{p3-2-first-order} gives us the results.
  % Changed 'Substituting' into 'Inserting'.
\end{proof}

Again, applying Wasow's theorem \ref{Wasow Thm}, we get the existence of the solutions to  $\textrm{P}_{\textrm{III}}^{(\textrm{ii})}$ with asymptotic expansions as given in \eqref{p3-2-expan}.
% Changed 'given' into 'as given'.

\begin{prop}
  In any sector of angle less than $\frac{3\pi}{2}$, there exists a solution $u(x)$ of $\textrm{P}_{\textrm{III}}^{(\textrm{ii})}$ whose asymptotic behaviour as $|x| \to \infty$ is given by $u_{f}^{(m)}(x)$ in \eqref{p3-2-formal} with $m=0,1,2$.
\end{prop}
\begin{proof}
  First we transform \eqref{p3-2-first-order} into the standard form \eqref{wasow model eqn} in Wasow's Theorem. Let $y=x^{\frac{1}{3}}$ and write
  % Changed 'introduce' into 'write'.
  $$u(x)=y(v(y)+a_{0}^{(m)}) \quad \textrm{and} \quad U(x)=\frac{1}{y^{2}}(V(y)+A_{0}^{(m)}).$$ Then $(v,V)$ satisfies the following equations:
  % Changed 'equations' into 'equations:'.
  \begin{equation} \label{p3-2-first-order-new}
    \frac{1}{y} \frac{d}{dy} \left(\begin{matrix} v \\ V \end{matrix}\right) = \left( \begin{matrix}
    3+\frac{2-3\beta}{y^{2}}(v+a_{0}^{(m)})+3(v+a_{0}^{(m)})^{2}(V+A_{0}^{(m)})\\
      3+\frac{3\beta-4}{y^{2}}(V+A_{0}^{(m)})-3(v+a_{0}^{(m)})(V+A_{0}^{(m)})^{2}
    \end{matrix} \right)
  \end{equation}
  and has a formal expansion given by
  \begin{equation} \label{p3-2-v&V}
    v_{f}^{(m)}(y)=\sum\limits_{n=1}^\infty a_{n}^{(m)}y^{-2n} \quad \textrm{and} \quad V_{f}^{(m)}(y)=\sum\limits_{n=1}^\infty A_{n}^{(m)}y^{-2n}.
  \end{equation}
   Again, let us denote the right-hand side of \eqref{p3-2-first-order-new} by $(f_1, f_2)^T$. Then we have
  \begin{align} \label{p3-2-J}
    \lim\limits_{y\rightarrow\infty}\left.\begin{pmatrix}
      \frac{\partial f_{1}}{\partial v} & \frac{\partial f_{1}}{\partial V} \\
      \frac{\partial f_{2}}{\partial v} & \frac{\partial f_{2}}{\partial V}
    \end{pmatrix}\right|_{v=V=0}=\begin{pmatrix}
      6 a_0^{(m)} A_0^{(m)} & 3(a_0^{(m)})^{2}\\
      -3(A_0^{(m)})^{2} & -6a_0^{(m)}A_0^{(m)}
    \end{pmatrix}:= J.
  \end{align}
  The eigenvalues of $J$ are
  \begin{equation} \label{p3-2-J-eig}
    \lambda_{1,2} = \pm 3\sqrt{3} \, e^{-\frac{2m \pi i}{3}} \neq 0, \qquad m = 0,1,2;
  \end{equation}
  see \eqref{p3-2-initial} for the values of $a_0^{(m)}$ and $ A_0^{(m)}$. Then, according to Wasow's theorem, there exist true solutions $(v(y),V(y))$ to \eqref{p3-2-first-order-new} in any sector of angle less than $\frac{\pi}{2}$ with given asymptotic expansion in \eqref{p3-2-v&V}. Recalling $y=x^{\frac{1}{3}}$ and the relations between $(u,U)$ and $(v,V)$, our proposition follows.
\end{proof}

As in Section \ref{sec-p3-1}, with the above results, we are ready to prove our Theorem \ref{thm-p3-2}.

\bigskip

\noindent\emph{Proof of Theorem \ref{thm-p3-2}.} Instead of studying the original system \eqref{p3-2-first-order}, it is more convenient to consider \eqref{p3-2-first-order-new}.
% Changed 'the new one \eqref{p3-2-first-order-new}' into '\eqref{p3-2-first-order-new} instead'.
Let $(v_0, V_0)$ be a solution of \eqref{p3-2-first-order} with asymptotic behaviour given in \eqref{p3-2-v&V}. Consider a perturbation of this solution
\begin{equation}
  v=v_0 + \widehat v \quad \textrm{and} \quad V=V_0 + \widehat V, \qquad |\widehat v|, |\widehat V| \ll 1.
\end{equation}
Then $\widehat v$ and  $\widehat V$ satisfy
% Changed 'satisfy the equation below' into 'satisfy'.
\begin{equation*}
  \frac{1}{y}\frac{d}{dy} \left(\begin{matrix} \widehat v \\ \widehat V \end{matrix}\right) = \left( \begin{matrix} (1-\beta+2xu_0 U_0) \widehat u + x u_{0}^{2} \widehat U  + 2x u_0 \widehat u \widehat U + xU_0 \widehat u^2 + x \widehat u^2 \widehat U \\
      x(1-U_{0}^{2}) \widehat u + (\beta-2-2xu_{0}U_{0}) \widehat U - 2x U_0 \widehat u \widehat U - x u_0 \widehat U^2 - x \widehat u\widehat U^2
    \end{matrix} \right).
\end{equation*}
Since $|\widehat v|, |\widehat V |\ll 1$, it is sufficient to consider the following linear system to determine the asymptotic behaviour of $\widehat v$ and $\widehat V$
% Changed '\widehat v, \widehat V ' into '|\widehat v|, |\widehat V |'
\begin{eqnarray}
  \frac{1}{y}\frac{d}{dy} \left(\begin{matrix} \widehat v \\ \widehat V \end{matrix}\right) &=& \begin{pmatrix}
      \frac{2-3\beta}{y^2}+6 (v_0+b_0^{(m)}) (V_0+B_0^{(m)}) & 3(v_0+b_0^{(m)})^{2} \\
      -3(V_0+B_0^{(m)})^{2} & \frac{3\beta-4}{y^2}-6(v_0+b_0^{(m)})(V_0+B_0^{(m)})
    \end{pmatrix} \left(\begin{matrix} \widehat v \\ \widehat V \end{matrix}\right) \nonumber \\
    &:= &B(y) \left(\begin{matrix} \widehat v \\ \widehat V \end{matrix}\right).
\end{eqnarray}
By \eqref{p3-2-v&V}, the coefficient matrix $B(y)$ in the above equation have the following asymptotic behaviour
\begin{equation} \label{p3-2-B}
B(y) = J +\mathcal{O}\left(\frac{1}{y^2}\right), \qquad \textrm{as } |y| \to \infty,
\end{equation}
where $J$ is the matrix in \eqref{p3-2-J} with eigenvalues given in \eqref{p3-2-J-eig}.
By the similar analysis as in the proof of Theorem \ref{thm-p3-1}, we have
\begin{equation}
  \widehat v(y) \sim c_1 y^{d_1} e^{\lambda_1 \frac{y^2}{2} } + c_2  y^{d_2} e^{\lambda_2 \frac{y^2}{2} }, \qquad \textrm{as } |y| \to \infty
\end{equation}
with two arbitrary parameters $c_1$ and $c_2$. To ensure our assumption $|\widehat v| \ll 1$, we need to choose either $c_i = 0$ or a correct sector such that $\re e^{ \frac{\lambda_i}{2} y^2} < 0$. According to the specific values of $\lambda_{1,2}$ in \eqref{p3-2-J-eig}, we should choose the sectors in the $y$-plane to be
\begin{align}
        \biggl\{y\in\mathbb{C}\; |\;|y|>|y_0|, (\frac{m}{3} - \frac{1}{4})\pi + \frac{k \pi}{2} <\arg{y}< (\frac{m}{3} + \frac{1}{4})\pi + \frac{k \pi}{2}  \biggr\}, \  \textrm{for }m=0,1,2,
\end{align}
where $k=0,1,2,3$ and $y_0 \neq 0$. As $y=x^{\frac{1}{3}}$, then the first part of our theorem is proved.

To prove the second part of the theorem, we can use similar analysis as in the proof of Theorem \ref{thm-p3-1} to show that, there exists a unique solution for \eqref{p3-2-first-order-new} in each of the following sectors with its asymptotic behaviour determined by the expansion in \eqref{p3-2-v&V}
\begin{align}
        \biggl\{y\in\mathbb{C}\; |\;|y|>|y_0|, (\frac{m}{3} - \frac{1}{4})\pi + \frac{k \pi}{2} <\arg{y}< (\frac{m}{3}+ \frac{3}{4})\pi + \frac{k \pi}{2}  \biggr\}, \  \textrm{for }m=0,1,2.
\end{align}
Note that the angles of the above $y$-sectors are $\pi$. Then, the angles of the corresponding $x$-sectors should be $3\pi$. So, for any branch of $x^{1/3}$ in the complex $x$-plane, we obtain a unique solution with the given asymptotic expansion in \eqref{p3-2-expan}. This completes our proof. \hfill $\Box$

\bigskip

%%%%%%%%%%%%%%%%%%%%%%%%%%%%%%%%%%%%%%%%%%%%%%%%%%%%%%%%%%%%%%%%%%%%%%%%%%%%%%%%%%%%%%

\section{The $\textrm{P}_{\textrm{IV}}$ equations} \label{sec-p4}

The Hamiltonian $\textnormal{H}_{\textnormal{IV}}$ for $\textrm{P}_{\textrm{IV}}$ is well-known in the literature
\begin{equation}
	\textnormal{H}_{\textnormal{IV}} = 2 u U^2- ( u^2 +2x u +2 \k )U+\K u
\end{equation}
(see for example \cite[p. 220]{Gromak-book}). So, \eqref{Painleve IV} can be written as the following system of first order differential equations:
\begin{align}\label{p4-first-order}
\begin{cases}
  \D\frac{du}{dx}=\frac{ \partial \textnormal{H}_{\textnormal{IV}} }
  {\partial U}=4uU-u^{2}-2xu-2\kappa_{0},  \vspace{2mm} \\
  \D\frac{dU}{dx}=-\frac{ \partial \textnormal{H}_{\textnormal{IV}} }
  {\partial u}
  =-2U^{2}+2uU+2xU-\kappa_{\infty},
  \end{cases}
\end{align}
where $\a=-\k+2\K+1$ and $\b=-2\k^2$. Then the formal solutions are obtained as follows.

\begin{prop} \label{prop-p4-formal}
  As $|x|\to \infty$, \eqref{p4-first-order} has
  four different formal solutions
  \begin{equation} \label{p4-formal-1}
   \textnormal{ Case 1:} \quad
   u_{1,f}(x)=x\sum\limits_{n=0}^{\infty}a_{n}^{(1)}x^{-2n} \qquad \textrm{and}
\qquad    U_{1,f}(x)=x\sum\limits_{n=0}^{\infty}A_{n}^{(1)}x^{-2n},
  \end{equation}
  where $a_{0}^{(1)}=-\frac23$, $a_{1}^{(1)}=\a$, $A_{0}^{(1)}=\frac13$, $A_{1}^{(1)}=\frac12-\k+\frac12 \K$, and the subsequent coefficients are given by
   \eqref{p4-recurrence-1}.
 \begin{equation} \label{p4-formal-2}
    \textnormal{ Case 2:} \quad
    u_{2,f}(x)=x\sum\limits_{n=0}^{\infty}a_{n}^{(2)}x^{-2n} \qquad \textrm{and}
\qquad    U_{2,f}(x)=\frac1x\sum\limits_{n=0}^{\infty}A_{n}^{(2)}x^{-2n},
  \end{equation}
  where $a_{0}^{(2)}=-2$, $a_{1}^{(2)}=-\a$, $A_{0}^{(2)}=-\frac12\K$, and the subsequent coefficients are given by \eqref{p4-recurrence-2}.
 \begin{equation} \label{p4-formal-3}
 \textnormal{ Case 3:} \quad
    u_{3,f}(x)=\frac1x\sum\limits_{n=0}^{\infty}a_{n}^{(3)}x^{-2n} \qquad \textrm{and}
\qquad    U_{3,f}(x)=x\sum\limits_{n=0}^{\infty}A_{n}^{(3)}x^{-2n},
  \end{equation}
  where $a_{0}^{(3)}=\k$,  $A_{0}^{(3)}=1$, $A_{1}^{(3)}=-\frac12 (1-2\k+\K)$,
  and the subsequent coefficients are given by \eqref{p4-recurrence-3}.
 \begin{equation} \label{p4-formal-4}
 \textnormal{ Case 4:} \quad
    u_{4,f}(x)=\frac1x\sum\limits_{n=0}^{\infty}a_{n}^{(4)}x^{-2n} \qquad \textrm{and}
\qquad    U_{4,f}(x)=\frac1x\sum\limits_{n=0}^{\infty}A_{n}^{(4)}x^{-2n},
  \end{equation}
 where $a_{0}^{(4)}=-\k$, $A_{0}^{(4)}=\frac12\K$, and the subsequent coefficients  are given by \eqref{p4-recurrence-4}.
\end{prop}

\begin{proof}
  Substituting \eqref{p4-formal-1}-\eqref{p4-formal-4} into \eqref{p4-first-order} gives us the results.
\end{proof}

\begin{rmk}
	When $\k=0$ in Case 3 and Case 4, we obtain the trivial solutions $u_{3,f}=0$ and $u_{4,f}=0$. Similarly, when $\K=0$ in Case 2 and Case 4, the trivial solutions $U_{2,f}=0$ and $U_{4,f}=0$ occur. We can exclude these trivial cases in our subsequent analysis.
\end{rmk}

Using Wasow's theorem again, we get the following result.

\begin{prop}
  In any sector of angle less than $\pi/2$, there exists a solution $u(x)$ of $\textnormal{P}_{\textnormal{IV}}$ whose asymptotic behaviour as $|x| \to \infty$ is given by $u_{mf}(x)$ in Proposition \ref{prop-p4-formal} with $m=1,2,3, 4$.
\end{prop}
\begin{proof}
 To arrive at the standard form \eqref{wasow model eqn} in Wasow's theorem,
 we consider the following change of variables.

  Case 1: Let $v(x)=x^{-1}u_{1,f}(x)-a_{0}^{(1)}$ and $V(x)= x^{-1}U_{1,f}(x)-A_0^{(1)}$. Then $(v,V)$ satisfies the following equations
  \begin{equation} \label{p4-first-order-1-new}
    \frac{1}{x}\frac{d}{dx} \left(\begin{matrix} v \\ V \end{matrix}\right) = \left( \begin{matrix}
      4(v+a_{0}^{(1)})(V+A_{0}^{(1)})-(v+a_{0}^{(1)})^{2}-2(v+a_{0}^{(1)})
      -\frac{1}{x^{2}}(v+a_{0}^{(1)})-\frac{2\kappa_{0}}{x^{2}}
       \\
       -2(V+A_{0}^{(1)})^{2}+2(v+a_{0}^{(1)})(V+A_{0}^{(1)})+2(V+A_{0}^{(1)})
       -\frac{1}{x^{2}}(V+A_{0}^{(1)})-\frac{\kappa_{\infty}}{x^{2}}
    \end{matrix} \right)
  \end{equation}
  and has a formal expansion given by
  \begin{equation}\label{v-formal}
    v_{f}(x)=\sum\limits_{n=1}^\infty a_{n} x^{-2n} \qquad \textrm{and} \qquad V_{f}(x)=\sum\limits_{n=1}^\infty A_{n} x^{-2n}
  \end{equation}
for constants $a_{n}$ and $A_{n}$. Let us denote the
right-hand side of \eqref{p4-first-order-1-new} by $(f_1, f_2)^{T}$, then we have
  \begin{align} \label{p4-J-1}
    \lim\limits_{x\rightarrow\infty}\left.\begin{pmatrix}
      \frac{\partial f_{1}}{\partial v} & \frac{\partial f_{1}}{\partial V} \\
      \frac{\partial f_{2}}{\partial v} & \frac{\partial f_{2}}{\partial V}
    \end{pmatrix}\right|_{v=V=0}=\begin{pmatrix}
      \frac23 & -\frac83 \\
      \frac23 & -\frac23
    \end{pmatrix}:= J_1.
  \end{align}
  The eigenvalues of $J_1$ are
 $
    \lambda_{1,2} = \pm \frac{2\sqrt 3}{3}i \neq 0.
 $
So by Wasow's theorem \ref{Wasow Thm}, we have, for any sector of angle smaller than $\pi/2$, there exist solutions $(u, U)$ of (\ref{p4-first-order}) with respective asymptotic behaviours $(u_{1,f}, U_{1,f})$ as described in \eqref{p4-formal-1}.
Next, we will prove the other cases.

Case 2:  Let $v(x)=x^{-1}u_{2,f}(x)-a_{0}^{(2)}$ and $V(x)= xU_{2,f}(x)-A_0^{(2)}$. Inserting these identities into \eqref{p4-first-order} then gives
\begin{equation*} \label{p4-first-order-2-new}
    \frac{1}{x}\frac{d}{dx} \left(\begin{matrix} v \\ V \end{matrix}\right) = \left( \begin{matrix}
     \frac{4}{x^{2}}(v+a_{0}^{(2)}) (V+A_{0}^{(2)})-(v+a_{0}^{(2)})^{2}-2(v+a_{0}^{(2)})
     -\frac{1}{x^{2}}(v+a_{0}^{(2)})-\frac{2\kappa_{0}}{x^{2}}
       \\
      -\frac{2}{x^{2}}(V+A_{0}^{(2)})^{2}+2(v+a_{0}^{(2)})(V+A_{0}^{(2)})
      +2(V+A_{0}^{(2)})+\frac{1}{x^{2}}(V+A_{0}^{(2)})-\kappa_{\infty}
    \end{matrix} \right)
  \end{equation*}
and $(v,V)$ has a formal expansion given by \eqref{v-formal}.

Case 3:  Let $v(x)=xu_{3,f}(x)-a_{0}^{(3)}$ and $V(x)= x^{-1}U_{3,f}(x)-A_0^{(3)}$. Then $(v,V)$ satisfies
\begin{equation*} \label{p4-first-order-3-new}
    \frac{1}{x}\frac{d}{dx} \left(\begin{matrix} v \\ V \end{matrix}\right) = \left( \begin{matrix}
     4(v+a_{0}^{(3)})(V+A_{0}^{(3)})-\frac{1}{x^{2}}(v+a_{0}^{(3)})^{2}-2(v+a_{0}^{(3)})
     +\frac{1}{x^{2}}(v+a_{0}^{(3)})-2\kappa_{0}
       \\
      -2(V+A_{0}^{(3)})^{2}+\frac{2}{x^{2}}(v+a_{0}^{(3)})(V+A_{0}^{(3)})+2(V+A_{0}^{(3)})-\frac{1}{x^{2}}(V+A_{0}^{(3)})-\frac{\kappa_{\infty}}{x^{2}}
    \end{matrix} \right)
  \end{equation*}
and has a formal expansion given by \eqref{v-formal}.

Case 4:  Let $v(x)=xu_{4,f}(x)-a_{0}^{(4)}$ and $V(x)= x U_{4,f}(x)-A_0^{(4)}$. Then $(v,V)$ satisfies
\begin{equation*} \label{p4-first-order-4-new}
    \frac{1}{x}\frac{d}{dx} \left(\begin{matrix} v \\ V \end{matrix}\right) = \left( \begin{matrix}
     \frac{4}{x^{2}}(v+a_{0}^{(4)})(V+A_{0}^{(4)})-\frac{1}{x^{2}}(v+a_{0}^{(4)})^{2}
     -2(v+a_{0}^{(4)})+\frac{1}{x^{2}}(v+a_{0}^{(4)})-2\kappa_{0}
       \\
      -\frac{2}{x^{2}}(V+A_{0}^{(4)})^{2}+\frac{2}{x^{2}}(v+a_{0}^{(4)})(V+A_{0}^{(4)})
      +2(V+A_{0}^{(4)})+\frac{1}{x^{2}}(V+A_{0}^{(4)})-{\kappa_{\infty}}
    \end{matrix} \right)
  \end{equation*}
  and has a formal expansion given by \eqref{v-formal}.

Following a similar manner, we then construct the Jacobian of $(f_1,f_2)^T$ evaluated at $(v,V)=0$ as $|x|\to \infty$,
\begin{align*}
    J_2=\begin{pmatrix}
      2 & 0 \\
      -\kappa_{\infty} & -2
    \end{pmatrix},
   && J_3=\begin{pmatrix}
      2 & 4\kappa_{0} \\
      0 & -2
    \end{pmatrix},
   &&
   J_4=\begin{pmatrix}
      -2 & 0 \\
      0 & 2
    \end{pmatrix}
 \end{align*}
for Case 2, Case 3 and Case 4, respectively. The eigenvalues of $J$ are
$
	\lambda_{1,2}=\pm 2.
$
Since each eigenvalue is different from zero, all the conditions of Wasow's theorem are fulfilled. Then our proposition follows from Wasow's theorem.
\end{proof}
Now, we are ready to prove our last theorem.

\begin{proof}[Proof of Theorem \ref{p4-thm}]
For simplicity, we only consider Case 1. Instead of studying \eqref{p4-first-order}, it is more convenient to consider \eqref{p4-first-order-1-new}. Let $(v_0, V_0)$ be a solution of \eqref{p4-first-order-1-new} with asymptotic behaviour given in \eqref{v-formal}. Consider a perturbation of this solution
\begin{equation}
  v=v_0 + \widehat v \quad \textrm{and} \quad V=V_0 + \widehat V, \qquad |\widehat v|, |\widehat V| \ll 1.
\end{equation}
Then $\widehat v$ and  $\widehat V$ satisfy the equations
\[
    \frac1x\frac{d\widehat v}{dx}= 4(v_0+\widehat v+a_{0}^{(1)}) \widehat V +4\widehat v
    (V_0+A_{0}^{(1)}) -\widehat v (2v_0+\widehat v+2a_{0}^{(1)}) -2 \widehat v
    -\frac{1}{x^2} \widehat v
    ,
\]
and
\[
    \frac1x\frac{d\widehat V}{dx}= -2\widehat V (2V_0+\widehat V+2A_{0}^{(1)})+2(v_0
    +\widehat v+a_{0}^{(1)}) \widehat V +2\widehat v
    (V_0+A_{0}^{(1)})  +2 \widehat V-\frac{1}{x^2} \widehat V.
\]
Since $|\widehat v|, |\widehat V| \ll 1$, it is sufficient to consider the following
linear system of differential equations to determine the asymptotic behaviour of $\widehat v$ and $\widehat V$
\begin{align*}
  \hspace{-35pt}\frac1x\frac{d}{dx} \left(\begin{matrix} \widehat v \\ \widehat V \end{matrix}\right) &=\begin{pmatrix}
     4
    (V_0+A_{0}^{(1)}) - 2(v_0+a_{0}^{(1)}) -2
    -\frac{1}{x^2}  & 4(v_0+a_{0}^{(1)}) \\
     2
    (V_0+A_{0}^{(1)})
    & -4 (V_0+A_{0}^{(1)})+2(v_0+a_{0}^{(1)})
    +2 -\frac{1}{x^2}
    \end{pmatrix} \left(\begin{matrix} \widehat v \\ \widehat V \end{matrix}\right)\\
    &:= B_1(x) \left(\begin{matrix} \widehat v \\ \widehat V \end{matrix}\right).
\end{align*}
By \eqref{p4-formal-1}, the coefficient matrix $B_1(x)$ in the above equation has the following asymptotic behaviour
\begin{equation} \label{p4-B-1}
B_1(x) = J_1 +\mathcal{O}\left(\frac{1}{x}\right), \qquad \textrm{as } |x| \to \infty,
\end{equation}
where $J_1$ is the matrix in \eqref{p4-J-1} with eigenvalues given by $\lambda_{1,2}=\pm\frac{2\sqrt 3}{3}i$.
By a similar analysis as the ones in the proofs of Theorem~1 and Theorem~2, we have
\begin{equation}
  \widehat v(x) \sim c_1 x^{d_1} e^{\lambda_1 \frac{x^2}{2} } + c_2  x^{d_2} e^{\lambda_2 \frac{x^2}{2} }, \qquad \textrm{as } |x| \to \infty
\end{equation}
with two arbitrary parameters $c_1$ and $c_2$. To ensure that $|\widehat v |\ll 1$, we need to choose either $c_i = 0$ or a correct sector such that $\re e^{\lambda_i \frac{x^2}{2} } < 0$. According to the specific values of $\lambda_{1,2}$, we should choose the sectors to be
\[
 S_{k}^{(1)}=
\left\{x\in\mathbb{C} \big | |x|>|x_0|, \frac{1}{2}k\pi<\arg{x}<\frac{1}{2}\pi+\frac{1}{2}k\pi\right\}.
\]

To prove the second part of our theorem and achieve uniqueness, we consider solutions in the following sectors
\begin{align*}
      S_{k, \pm\varepsilon}^{(1)} =
       \biggl\{x\in\mathbb{C}\; |\;|x|>|x_0|,  \frac12 k \pi \pm\varepsilon <\arg{x}< \frac{\pi}{2} + \frac12 k \pi \pm\varepsilon \biggr\}
\end{align*}
with a small $\varepsilon>0$.
Let $(u_1, U_1)$ and $(u_2, U_2)$ be two solutions to \eqref{p4-first-order} whose asymptotic expansions are given by \eqref{p4-formal-1} in sectors $ S_{k, \varepsilon}^{(1)}$ and $ S_{k+1, -\varepsilon}^{(1)}$, respectively. Consider the difference of these solutions
\begin{equation}
  w:= u_1-u_2 \quad \textrm{and} \quad W:=U_1-U_2,
\end{equation}
which are defined in the overlap of the two sectors
\begin{equation}
  \widehat S_{k,\varepsilon}^{(1)} = S_{k, \varepsilon}^{(1)} \cap  S_{k+1, -\varepsilon}^{(1)}.
\end{equation}
Since $(u_1, U_1)$ and $(u_2, U_2)$ satisfy the same asymptotic expansion, we have
\begin{equation} \label{p4-w-asy-1}
  w(x) = O(x^{-j}) \quad \textrm{and} \quad W(x) = O(x^{-j}) \qquad \textrm{as } |x| \to \infty \textrm{ and } x \in \widehat S_{k,\varepsilon}^{(1)}
\end{equation}
for all $j \in \mathbb{N}$. Moreover, one can verify that $w(x)$ and $W(x)$ satisfy
\begin{align*}
 \frac1x \frac{d}{dx} \left(\begin{matrix} w \\  W \end{matrix}\right) &=\begin{pmatrix}
     4U_1 -2u_1-2 & 4u_1 \\
     2U_1
    & -4 U_1 +2 u_1+2
    \end{pmatrix} \left(\begin{matrix} w \\ W \end{matrix}\right)
     := \widetilde B_1(x) \left(\begin{matrix} w \\ W \end{matrix}\right).
\end{align*}
The coefficient matrix $\widetilde B_1(x)$ satisfies the same asymptotic behaviour as $B_1(x)$ does in \eqref{p4-B-1}. Then by the same argument as in the first part, we have
\begin{equation}
  w(x) \sim c_1 x^{d_1} e^{\lambda_1 \frac{x^2}{2} } + c_2  x^{d_2} e^{\lambda_2 \frac{x^2}{2} }, \qquad \textrm{as } |x| \to \infty.
\end{equation}
To ensure that $w(x)$ satisfies the asymptotic behaviour described in \eqref{p4-w-asy-1} for $x \in\widehat S_{k,\varepsilon}^{(1)}$, we must choose $c_1=c_2 =0$ in the above formula. This means $u_1 = u_2$ in $\widehat S_{k,\varepsilon}^{(1)}$ and the sector of validity can be extended to $ \Omega_{k}^{(1)}$ in Theorem~\ref{p4-thm}. The other cases can be handled in a similar manner. This completes the proof.
\end{proof}

%\section{Discussion}

\subsection*{Acknowledgements}

  We would like to thank Professors Nalini Joshi, Alexander Kitaev and Yoshitsugu Takei for their helpful comments and discussions.

  DD and PT were partially supported by a grant from the Research Grants Council of the Hong Kong Special Administrative Region, China (Project No. CityU 100910). In addition, DD was partially supported partially supported by a grant from the City University of Hong Kong (Project No. 7002883).

%%%%%%%%%%%%%%%%%%%%%%%%%%%%%%%%%%%%%%%%%%%%%%%%%%%%%%%%%%%%%%%%%%%%%%%%%%%%%%%%%%%%%%%%%%

\end{document}